\documentclass[12pt]{amsart}
\usepackage{SSdefn}
\usepackage[lite,nobysame,alphabetic]{amsrefs}

\author{Daniel Erman}
\address{Department of Mathematics, University of Wisconsin, Madison, WI}
\email{\href{mailto:derman@math.wisc.edu}{derman@math.wisc.edu}}
\urladdr{\url{http://math.wisc.edu/~derman/}}

\author{Steven V Sam}
\address{Department of Mathematics, University of California, San Diego, CA}
\email{\href{mailto:ssam@ucsd.edu}{ssam@ucsd.edu}}
\urladdr{\url{http://math.ucsd.edu/~ssam/}}

\author{Andrew Snowden}
\address{Department of Mathematics, University of Michigan, Ann Arbor, MI}
\email{\href{mailto:asnowden@umich.edu}{asnowden@umich.edu}}
\urladdr{\url{http://www-personal.umich.edu/~asnowden/}}

\thanks{DE was partially supported by NSF DMS-1601619.  SS was partially supported by NSF DMS-1500069 and DMS-1651327 and a Sloan Fellowship.  AS was supported by NSF DMS-1303082 and DMS-1453893 and a Sloan Fellowship.}

\subjclass[2010]{%
14M10, 
13C40,  
13L05. 
}

\begin{document}

\title{Strength and Hartshorne's Conjecture in high degree}
\maketitle

\section{Introduction}
Hartshorne conjectured that every smooth, codimension $c$ subvariety of $\bP^n$, with $c<\frac{1}{3}n$, is a complete intersection~\cite[p.~1017]{hartshorne-bulletin}.  We give a new proof of the case when $n\gg \deg X$.  If $X$ is smooth, then we say that the codimension of the singular locus of $X$ is $\dim X+1$.
\begin{theorem}\label{thm:more general}
There is a function $N(c,e)$ such that if $X\subseteq \bP^n_{\bk}$ is an equidimensional, projective subscheme of codimension $c$ and 
degree~$e$, and if the singular locus of $X$ has codimension at least $N(c,e)$, then $X$ is a complete intersection.  In particular, the function $N(c,e)$ does not depend on $n$ or on the field $\bk$.
\end{theorem}

Results like Theorem~\ref{thm:more general} have a long history.  In characteristic zero, Hartshorne first proved the above result in~\cite[Theorem 3.3]{hartshorne-bulletin}.  In parallel, and also in characteristic zero, Barth and Van de Ven proved an effective version of this result, showing that $N=\frac{5}{2}e(e-7)+c$ works~\cite{barth-icm}.
Huneke's work in \cite[Theorem~1.1]{huneke} also implies Theorem~\ref{thm:more general} for arbitrary $c$ and over fields of arbitrary characteristic, and gives an explicit bound in terms of the analytic spread.  Further improvements include: work of Ran~\cite[Theorem]{ran}, sharpening the bound and extending it to arbitrary characteristic, but only when $c=2$; and Bertram-Ein-Lazarsfeld's \cite[Corollary~3]{bertram-ein-lazarsfeld}, which sharpens the bound in arbitrary codimension, but only holds in characteristic zero.  See also~\cites{ballico-chiantini, holme-schneider}.

The novelty in our work is through the new, and concise, method of proof.  The main ingredient in ~\cite{barth-icm} is an analysis of the variety of lines in $X$ through a point, and many of the aforementioned proofs make use of Kodaira Vanishing and topological results like Lefschetz-type restriction theorems. By contrast, we derive Theorem~\ref{thm:more general} from elementary consequences of the notion of strength introduced in~\cite{ananyan-hochster}, and of our result in ~\cite{ess-stillman} that the graded ultraproduct of polynomial rings is isomorphic to a polynomial ring.   This explicitly connects Hartshorne's Conjecture with the circle of ideas initiated by Ananyan and Hochster in~\cite{ananyan-hochster} in their proof of Stillman's Conjecture.  It also underscores similarities between the two conjectures, both of which propose limits on the possible behaviors for varieties of codimension $c$ in $\bP^n$ when $n\gg c$.

The ideas of~\cite{ananyan-hochster} strongly motivated this work, as the connection between strength and the codimension of the singular locus is one of the central ideas in that paper.  Our viewpoint also has some overlap with the Babylonian tower theorems, like ~\cite[Theorems~I and IV]{BVdV} and those in~\cite{coandua, flenner, sato} among others. From an algebraic perspective, the natural setting for such statements is an inverse limit of polynomial rings, and~\cite{ess-stillman, ess-imperfect} shows such an inverse limit shares many properties with the ultraproduct ring.

\begin{remark}
A classical result (see e.g.~\cite{eisenbud-harris}) shows that for any nondegenerate, integral variety in $\bP^n$, the codimension is at most its degree. Thus, under these hypotheses, Theorem~\ref{thm:more general} could be rephrased so as to remove the dependence on $c$.
\end{remark}

\section*{Acknowledgments}
We thank Craig Huneke, David Kazhdan, and an anonymous referee for useful comments.
 
\section{Setup and Background} \label{sec:background}
Each closed subscheme $X\subseteq \bP^n_\bk$ determines a homogeneous ideal $I_X\subseteq \bk[x_0,\dots,x_n]$.  The scheme $X$, or the ideal $I_X$, is {\bf equidimensional of codimension $c$} if all associated primes of $I_X$ have codimension $c$, and $X$ is a {\bf complete intersection} if $I_X$ is defined by a regular sequence.  Since the minimal free resolution of $I_X$ is stable under extending the ground field $\bk$, the property of being a complete intersection is also stable under field extension.

From here on, $\bk$ and $\bk_i$ will denote fields. If $R$ is a graded ring with $R_0=\bk$, then as in~\cite{ananyan-hochster}, we define the {\bf strength} of a homogeneous element $f\in R$ to be the minimal integer $k \ge -1$ for which there is a decomposition $f=\sum_{i=1}^{k+1} g_i h_i$ with $g_i$ and $h_i$ homogeneous elements of $R$ of positive degree, or $\infty$ if no such decomposition exists. The {\bf collective strength} of a set of homogeneous elements $f_1,\dots,f_r\in R$ is the minimal strength of a non-trivial homogeneous $\bk$-linear combination of the $f_i$.

 \begin{lemma}\label{lem:decreasing strength}
Let $R$ be a graded ring with $R_0=\bk$.  If $I\subseteq R$ is homogeneous and finitely generated, then $I$ has a generating set  of homogeneous elements $f_1,\dots,f_r$ where the strength of $f_k$ equals the collective strength of $f_1,\dots,f_k$ for each $1\leq k \leq r$.
\end{lemma}

 \begin{proof}
 Choose any homogeneous generators $g_1,\dots,g_r$ of $I$.  We prove the statement by induction on $r$.  For $r=1$ the statement is tautological.  Now let $r>1$.  By definition of collective strength, we have a $\bk$-linear combination $f_r=\sum_{i=1}^r a_ig_i$ such that the strength of $f_r$ equals the collective strength of $g_1,\dots,g_r$.  After relabeling, we can assume that $a_r\ne 0$ and it follows that $g_1,\dots,g_{r-1},f_r$ generate $I$.  Applying the induction hypothesis to the ideal $(g_1,\dots,g_{r-1})$ yields the desired result.
 \end{proof}

Let $Q=(f_1,\dots,f_r) \subseteq \bk[x_1,x_2,\dots]$.  The ideal of $c\times c$ minors of the Jacobian matrix of $(\frac{\partial f_i}{\partial x_j})$ does not depend on the choice of generators of $Q$.  We denote this ideal by $J_c(Q)$.

\begin{lemma}\label{lem:strength and J}
Let $Q=(f_1,\dots,f_r)$ be a homogeneous ideal in $\bk[x_1,x_2,\dots]$.  If the strength of $f_i$ is at most $s$ for $c\leq i \leq r$,  then $\codim J_c(Q)\leq (r-c+1)(2s+2)$.
\end{lemma}

\begin{proof}
For each $c \leq i \leq r$, we write $f_i=\sum_{j=0}^{s} a_{i,j}h_{i,j}$ where  $a_{i,j}$ and $h_{i,j}$ have positive degree for all $i,j$.  Write $L_i$ for the ideal $(a_{i,j}, h_{i,j} \mid 0\leq j \leq s)$ and let $L =L_c+L_{c+1}+\dots+L_r$.  The $i$th row of the Jacobian matrix has entries $\frac{\partial f_i}{\partial x_k}$; thus by the product rule, every entry in this row is in $L_i$.  Since every $c\times c$ minor of the Jacobian matrix will involve row $i$ for some $c\leq i \leq r$, it follows that $J_c(Q) \subseteq L$.  Thus $\codim J_c(Q)\leq \codim L$, which by the Principal Ideal Theorem is at most $(r-c+1)(2s+2)$, as this is the number of generators of $L$.
\end{proof}

We briefly recall the definition of the ultraproduct ring, referring to~\cite[\S4.1]{ess-stillman} for a more detailed discussion.  Let $\cI$ be an infinite set and let $\cF$ be a non-principal ultrafilter on $\cI$.  We refer to subsets of $\cF$ as {\bf neighborhoods of $\ast$}, where $\ast$ is an imaginary point of $\cI$. For each $i\in \cI$, let $\bk_i$ be an infinite perfect field.  The ultraproduct of the $\{\bk_i\}$ consists of collections $c=(c_i)_{i\in \cI}$ where $c_i\in \bk_i$, modulo the relation that $c=0$ if and only if $c_i=0$ for all $i$ in some neighborhood of $\ast$; by the axioms of ultrafilters, this ultraproduct is also a (perfect) field. Let $\bS$ denote the graded ultraproduct of $\{\bk_i[x_1,x_2,\dots]\}$, where each polynomial ring is given the standard grading.  An element $g\in \bS$ of degree $d$ corresponds to a collection $(g_i)_{i\in \cI}$ of degree $d$ elements $g_i\in \bk_i[x_1,x_2,\dots]$, modulo the relation that $g=0$ if and only if $g_i=0$ for all $i$ in some neighborhood of $\ast$.  For a homogeneous $g\in \bS$ we write $g_i$ for the corresponding element in $\bk_i[x_1,x_2,\dots]$, keeping in mind that this is only well-defined for $i$ in some neighborhood of $\ast$.  The following comes from~\cite[Theorems~1.2 and 4.6]{ess-stillman}:
\begin{theorem}\label{thm:ultra poly}
Let $K$ be the ultraproduct of perfect fields $\{\bk_i\}$ and fix $y_1,\dots, y_c\in \bS$ of infinite collective strength.  There is a set $\cU$, containing the $y_i$, such that
$\bS$ is isomorphic to the polynomial ring $K[\cU]$.
\end{theorem}

The following result follows immediately from~\cite[Theorem~5.2]{cmpv}.  While that result does not explicitly note the independence on the field $\bk$, it follows from the proof.

\begin{lemma}\label{lem:numgens and degree}
Fix $c$ and $e$.  There exist positive integers $d$ and $r$, depending only on $c$ and $e$, such that any homogeneous, equidimensional, and radical ideal $Q\subseteq \bk[x_1,\dots,x_n]$ of codimension $c$ and degree $\leq e$ can be generated (not necessarily minimally) by homogeneous polynomials $g_1,\dots,g_r$ where $\deg(g_i) \le d$. Neither $r$ nor $d$ depend on $n$ or $\bk$.
\end{lemma}
\begin{proof}
By~\cite[Theorem~5.2]{cmpv}, both the regularity of $Q$ and the individual Betti numbers $\beta_{i,j}(Q)$ are bounded solely in terms of $c$ and $e$.  Choosing $d$ as the regularity bound and $r$ as the bound on $\sum_{i=1}^d \beta_{0,d}(Q)$, we obtain the desired statement.
\end{proof}
 
 \section{Proof of the main result}

\begin{theorem}\label{thm:Nreg}
There is a function $N(c,e)$ such that if $Q\subseteq \bk[x_1,\dots,x_n]$ is a homogeneous, equidimensional ideal of codimension $c$ and degree $e$ and if $V(Q)$ is nonsingular in codimension $\geq N(c,e)$, then $Q$ is defined by a regular sequence of length $c$.  In particular, $N(c,e)$ does not depend on $n$ or on the field $\bk$.
\end{theorem}
\begin{remark}
Since an equidimensional ideal of codimension $c$ that is nonsingular in codimension $2c+1$ must be prime, it would be equivalent to rephrase Theorem~\ref{thm:Nreg} in terms of prime ideals.  We stick with equidimensional and radical ideals because some of the auxiliary results in this paper might be of interest with this added generality.
\end{remark}

\begin{proof}[Proof of Theorem~\ref{thm:Nreg}]
We first reduce to the case where $\bk$ is perfect.  Extending the field will change neither the minimal number of generators of $Q$, nor the codimension of the singular locus.  By taking $N(c,e)\geq 1$, we can also assume that $Q$ is radical, even after extending the field.  Finally, since a field extension will not change the codimension of any minimal prime of $Q$~\cite[00P4]{stacks}, we can assume that $\bk$ is perfect and that $Q$ is radical and equidimensional of codimension $c$.

Suppose that the theorem were false.  Then for some fixed $c,e$ and for each $j\in \bN$ we can find an equidimensional, radical ideal $Q_j'\subseteq \bk_j[x_1,x_2,\dots]$ (with $\bk_j$ perfect) of codimension $c$ and degree $e$ that is not a complete intersection, but where the codimension of the singular locus of $V(Q_j')$ tends to $\infty$ as $j\to \infty$.  Since the singular locus of $V(Q_j')$ is defined by $Q_j'+J_c(Q_j')$, this implies that $\codim J_c(Q_j') \to \infty$ as $j\to \infty$.  We choose a function $m \colon \cI \to \bN$ where $m(i)$ is unbounded in each neighborhood of $\ast$.  For each $i \in \cI$, define $Q_i$ to be any of the $Q_j'$ satisfying $\codim J_c(Q_j')\geq m(i)$.  By construction, $\codim J_c(Q_i)$ is unbounded in every neighborhood of $\ast$.

We now apply Lemma~\ref{lem:numgens and degree} for each $i\in \cI$, to find positive integers $r$ and $d$ and homogeneous $g_{1,i}, \dots, g_{r,i}$ of degree $\leq d$ which generate $Q_i$.  Let $g_1=(g_{1,i}), \dots, g_r=(g_{r,i})$ be the corresponding elements in $\bS$ and let $Q=(g_1,\dots,g_r)$.  By Lemma~\ref{lem:decreasing strength}, we can find a new homogeneous generating set $f_1,\dots,f_r$ of $Q$ where the strength of $f_k$ is the collective strength of $f_1,\dots, f_k$ for each $1\leq k\leq r$.  For each $k$, we may write $f_k = (f_{k,i})$.

If $f_c$ had strength at most $s$, then we observe that the same holds for $f_{c,i}$ in a neighborhood of $\ast$; for if $f_c=\sum_{j=0}^s a_jh_j$ then $f_{c,i}=\sum_{j=0}^s (a_j)_i(f_j)_i$ for $i$ near $\ast$.  But by Lemma~\ref{lem:strength and J}, this would imply that $\codim J_c(Q_i)$ is bounded in a neighborhood of $\ast$.  Since this cannot happen, $f_c$ must have infinite strength.  Thus the collection $f_1,\dots,f_c$ has infinite collective strength and so applying Theorem~\ref{thm:ultra poly} with $y_i=f_i$, we conclude that $f_1,\dots,f_c$ are independent variables in $\bS$.  In particular, $(f_1,\dots,f_c)$ defines a prime ideal of codimension $c$ and we therefore must have $f_{c+1}=\dots=f_{r}=0$.  By~\cite[Corollary~4.10]{ess-stillman}, there is a neighborhood of $\ast$ where each $Q_i$ is a complete intersection, contradicting our original assumption.
\end{proof}

\begin{proof}[Proof of Theorem~\ref{thm:more general}]
As in the beginning of the proof of Theorem~\ref{thm:Nreg}, we can quickly reduce to the case where $\bk$ is perfect.  For a fixed $c$ and $e$, we let $N$ equal the bound from Theorem~\ref{thm:Nreg}. Fix some $X\subseteq \bP^n$ satisfying the hypotheses of Theorem~\ref{thm:more general}, and let $Q\subseteq \bk[x_1,\dots,x_{n+1}]$ be the defining ideal of $X$.  By Theorem~\ref{thm:Nreg}, $Q$ is defined by a regular sequence, and thus $X$ is a complete intersection.
\end{proof}

\end{document}